\documentclass{article}%
\usepackage{graphicx}
\usepackage{amsmath}
\usepackage{amsfonts}
\usepackage{amssymb}%
\providecommand{\U}[1]{\protect\rule{.1in}{.1in}}
\providecommand{\U}[1]{\protect\rule{.1in}{.1in}}
\newtheorem{theorem}{Theorem}

\newtheorem{corollary}[theorem]{Corollary}

\newtheorem{definition}[theorem]{Definition}
\newtheorem{example}[theorem]{Example}

\newtheorem{notation}[theorem]{Notation}

\newtheorem{proposition}[theorem]{Proposition}
\newtheorem{remark}[theorem]{Remark}

\newenvironment{proof}[1][Proof]{\textbf{#1.} }{\ \rule{0.5em}{0.5em}}
\begin{document}

\title{Weil Diffeology I:\\Classical Differential Geometry}
\author{Hirokazu NISHIMURA\\Institute of Mathematics, University of Tsukuba\\Tsukuba, Ibaraki 305-8571\\Japan}
\maketitle

\begin{abstract}
Topos theory is a category-theoretical axiomatization of set theory. Model
categories are a category-theoretical framework for \textit{abstract} homotopy
theory. They are complete and cocomplete categories endowed with three classes
of morphisms (called \textit{fibrations}, \textit{cofibrations} and
\textit{weak} \textit{equivalences)} satisfying certain axioms. We would like
to present an abstract framework for classical differential geometry as an
extension of topos theory, hopefully comparable with model categories for
homotopy theory. Functors from the category $\mathfrak{W}$\ of \textit{Weil
algebras} to the category $\mathbf{Sets}$\ of sets are called \textit{Weil
spaces} by Wolfgang Bertram and form the \textit{Weil topos} after Eduardo J.
Dubuc. The Weil topos is endowed intrinsically with the \textit{Dubuc
functor}, a functor from a larger category $\widetilde{\mathfrak{W}}$\ of
\textit{cahiers algebras} to the Weil topos standing for the incarnation of
each algebraic entity of $\widetilde{\mathfrak{W}}$\ in the Weil topos. The
Weil functor and the canonical ring object are to be defined in terms of the
Dubuc functor. The principal objective in this paper is to present a
category-theoretical axiomatization of the Weil topos with the Dubuc functor
intended to be an adequate framework for axiomatic classical differential
geometry. We will give an appropriate formulation and a rather complete proof
of a generalization of the familiar and desired fact that the tangent space of
a \textit{microlinear} Weil space is a module over the canonical ring object.

\end{abstract}

\section{\label{s0}Introduction}

Differential geometry usually exploits not only the techniques of
differentiation but also those of integration. In this paper we would like to
use the term "\textit{differential geometry}" in its literal sense, that is,
\textit{genuinely} differential geometry, which is vast enough as to encompass
a large portion of the theory of connections and the core of the theory of Lie
groups. Now we know well that there is a horribly deep and overwhelmingly
gigantic valley between differential calculus of the 17th and 18th centuries
(that is to say, that of the good old days of Newton, Leibniz, Lagrange,
Laplace, Euler and so on) and that of our modern age since the 19th century
when Angustin Louis Cauchy was active. The former exquisitely resorts to
nilpotent infinitesimals, while the latter grasps differentiation in terms of
limits by using so-called $\varepsilon-\delta$ arguments formally.
Differential geometry based on the latter style of differentiation is
generally called \textit{smootheology}, while we propose that differential
geometry based on the former style of differentiation might be called
\textit{Weilology}.

As is well known, the category of topological spaces and continuous mappins is
not cartesian closed. The classical example of a \textit{convenient category}
of topological spaces for working topologisits was suggested by Norman
Steenrod \cite{ste1}\ in the middle of the 1960s, namely, the category of
compactly generated spaces. Now the category of finite-dimensional smooth
manifolds and smooth mappings is not cartesian closed, either. Convenient
categories for smootheology have been proposed by several authors in several
corresponding forms. Among them Souriau's \cite{so1}\ approach based upon the
category $\mathfrak{O}$\ of open subsets $O$'s\ of $\mathbf{R}^{n}$'s\ and
smooth mappings between them has developed into a galactic volume of
\textit{diffeology}, for which the reader is referred to \cite{ig1}. A
\textit{diffeological space} is a set $X$\ endowed with a subset
$\mathcal{D}\left(  O\right)  \subseteq X^{O}$\ for each $O\in\mathfrak{O}%
$\ such that, for any morphism $f:O\rightarrow O^{\prime}$ in $\mathfrak{O}%
$\ and any $\gamma\in\mathcal{D}\left(  O^{\prime}\right)  $, we have
$\gamma\circ f\in\mathcal{D}\left(  O\right)  $. A diffeological map between
diffeological spaces $\left(  X,\mathcal{D}\right)  $\ and $\left(
X,\mathcal{D}^{\prime}\right)  $\ is a mapping $f:X\rightarrow X^{\prime}%
$\ such that, for any $O\in\mathfrak{O}$\ and any $\gamma\in\mathcal{D}\left(
O\right)  $, we have $f\circ\gamma\in\mathcal{D}^{\prime}\left(  O\right)  $.

Roughly speaking, there are two approaches to geometry in representing spaces,
namely, \textit{contravariant} (\textit{functional}) and \textit{covariant}
(\textit{parameterized}) ones, for which the reader is referred, e.g., to
Chapter 3 of \cite{pa1}\ as well as \cite{ni9} and \cite{ni10}. Diffeology
finds itself in the covariant realm. The contravariant approach boils down
spaces to their function algebras. We are now accustomed to admitting
\textit{all} algebras to stand for \textit{abstract} spaces in some way or
other, whatever they may be. This is a long tradition of algebraic geometry
since as early as Alexander Grothendieck. Now we are ready to acknowledge any
functor $\mathfrak{O}^{\mathrm{op}}\rightarrow\mathbf{Sets}$\ as an
\textit{abstract }diffeological space. Then it is pleasant to enjoy

\begin{theorem}
The category of abstract diffeological spaces and natural transformations
between them is a topos.
\end{theorem}

Turning to Weilology, a space should be represented as a functor
$\mathbf{Inf}^{\mathrm{op}}\rightarrow\mathbf{Sets}$, where $\mathbf{Inf}%
$\ stands for the category of nilpotent infinitesimal spaces. Since our creed
tells us that the category $\mathbf{Inf}^{\mathrm{op}}$\ is equivalent to
$\mathfrak{W}$, a space should be no other than a functor $\mathfrak{W}%
\rightarrow\mathbf{Sets}$, for which Wolfgang Bertram \cite{ber3}\ has coined
the term "\textit{Weil space}". To be sure, we have

\begin{theorem}
The category of Weil spaces and natural transformations between them is a topos.
\end{theorem}

\section{\label{s1}Cahiers Algebras}

Unless stated to the contrary, our base field is assumed to be $\mathbf{R}%
$\ (real numbers) throughout the paper, so that we will often say "Weil
algebra" simply in place of Weil $\mathbf{R}$-algebra". For the exact
definition of a Weil algebra, the reader is referred to \S I.16 of \cite{koc1}.

\begin{notation}
We denote by $\mathfrak{W}$\ the category of Weil algebras.
\end{notation}

\begin{remark}
$\mathbf{R}$ is itself a Weil algebra, and it is an initial object in the
category $\mathfrak{W}$.
\end{remark}

\begin{definition}
An $\mathbf{R}$-algebra isomorphic to an $\mathbf{R}$-algebra of the form
$\mathbf{R}\left[  X_{1},...,X_{n}\right]  \otimes W$ with $\mathbf{R}\left[
X_{1},...,X_{n}\right]  \ $being the polynomial algebra over $\mathbf{R}\ $in
indeterminates $X_{1},...,X_{n}$ (possibly $n=0$, when the definition
degenerates to Weil algebras) and $W$\ being a Weil algebra is called a
\underline{cahiers algebra}.
\end{definition}

\begin{remark}
This definition of a cahiers algebra is reminiscent of that in the definition
of Cahiers topos, where we consider a product of a Cartesian space
$\mathbf{R}^{n}$\ and a formal dual of a Weil algebra.
\end{remark}

\begin{notation}
We denote by $\widetilde{\mathfrak{W}}$\ the category of cahiers algebras.
\end{notation}

\begin{remark}
The category $\mathfrak{W}$\ is a full subcategory of the category
$\widetilde{\mathfrak{W}}$. Both are closed under the tensor product
\ $\otimes$.
\end{remark}

\begin{notation}
We will use such a self-explanatory notation as $Z\rightarrow X/\left(
X^{2}\right)  $\ or $X/\left(  X^{2}\right)  \leftarrow Z$ for the morphism
$\mathbf{R}\left[  Z\right]  \rightarrow\mathbf{R}\left[  X\right]  /\left(
X^{2}\right)  $ assigning $X$\ modulo $\left(  X^{2}\right)  $ to $Z$.
\end{notation}

\section{\label{s2}Weil Spaces}

\begin{definition}
A \underline{Weil space} is simply a functor $F$\ from the category
$\mathfrak{W}$\ of Weil algebras to the category $\mathbf{Sets}$\ of sets. A
\underline{Weil morphism} from a Weil space $F$\ to another Weil space $G$\ is
simply a natural transformation from the functor $F$ to the functor $G$.
\end{definition}

\begin{remark}
The term "Weil space" has been coined in \cite{ber3}.
\end{remark}

\begin{example}
The Weil prolongation of a "manifold" in its broadest sense (cf. \cite{ber1})
by a Weil algebra was fully discussed by Bertram and Souvay, for which the
reader is cordially referred to \cite{ber2}. We are happy to know that any
manifold naturally gives rise to its associated Weil space, which can be
regarded as a functor from the category of manifolds to the category
$\mathbf{Weil}$. It should be stressed without exaggeration that the functor
is not full in general, for which the reader is referred to exuberantly
readable \S 1.6 (discussion) of \cite{ber3}.
\end{example}

\begin{example}
The Weil prolongation $A\otimes W$\ of a $C^{\infty}$-algebra $A$\ by a Weil
algebra $W$\ was discussed in Theorem III.5.3 of \cite{koc1}. We are happy to
know that any $C^{\infty}$-algebra naturally gives rise to its associated Weil space.
\end{example}

\begin{notation}
We denote by $\mathbf{Weil}$\ the category of Weil spaces and Weil morphisms.
\end{notation}

\begin{remark}
Dubuc \cite{du1}\ has indeed proposed the topos $\mathbf{Weil}$\ as the first
step towards the \textit{well adapted} model theory of \textit{synthetic
differential geometry}, but we would like to contend somewhat radically that
the topos $\mathbf{Weil}$\ is verbatim the \textit{central object of study in
classical differential geometry}
\end{remark}

It is well known (cf. Chapter 1 of \cite{ma1}) that

\begin{theorem}
The category $\mathbf{Weil}$\ is a topos. In particular, it is locally
cartesian closed.
\end{theorem}

\begin{remark}
Dubuc \cite{du1} has called the category $\mathbf{Weil}$ the \underline{Weil
topos}.
\end{remark}

\begin{remark}
The category of Fr\"{o}licher spaces is indeed cartesian closed, but it is not
locally cartesian closed. On the other hand, the category of diffeological
spaces is locally cartesian closed. For these matters, the reader is referred
to \cite{st1}. It was shown by Baez and Hoffnung \cite{ba1}\ that
diffeological spaces as well as Chen spaces are no other than concrete sheaves
on concrete sites.
\end{remark}

\begin{definition}
The \underline{Weil prolongation} $F^{W}$\ of a Weil space $F$\ by a Weil
algebra $W$\ is simply the composition of the functor $\left(  \_\right)
\otimes W:\mathfrak{W}\rightarrow\mathfrak{W}$ and the functor $F:\mathfrak{W}%
\rightarrow\mathbf{Sets}$, namely%
\[
F\left(  \left(  \_\right)  \otimes W\right)  :\mathfrak{W}\rightarrow
\mathbf{Sets}%
\]
which is surely a Weil space.
\end{definition}

\begin{remark}
$\left(  \_\right)  ^{\left(  \cdot\right)  }$ assigning $F^{W}$\ to each
$\left(  W,F\right)  \in\mathfrak{W}\times\mathbf{Weil}$\ can naturally be
regarded as a bifunctor $\mathfrak{W}\times\mathbf{Weil}\rightarrow
\mathbf{Weil}$.
\end{remark}

Trivially we have

\begin{proposition}
\label{p2.1}For any Weil space $F$\ and any Weil algebras $W_{1}$\ and $W_{2}%
$, we have%
\[
\left(  F^{W_{1}}\right)  ^{W_{2}}=F^{W_{1}\otimes W_{2}}%
\]

\end{proposition}

\begin{remark}
The so-called Yoneda embedding%
\[
\mathbf{y}:\mathfrak{W}^{\mathrm{op}}\rightarrow\mathbf{Weil}%
\]
is full and faithful. The famous Yoneda lemma claims that%
\begin{equation}
F\left(  \_\right)  \cong\mathrm{Hom}_{\mathbf{Weil}}\left(  \mathbf{y}\left(
\_\right)  ,F\right)  \label{2.1}%
\end{equation}
for any Weil space $F$. The Yoneda embedding can be extended to%
\[
\widetilde{\mathbf{y}}:\widetilde{\mathfrak{W}}^{\mathrm{op}}\rightarrow
\mathbf{Weil}%
\]
by%
\[
\widetilde{\mathbf{y}}\left(  A\right)  =\mathrm{Hom}_{\mathbf{R-Alg}}\left(
A,\_\right)
\]
for any $A\in\widetilde{\mathfrak{W}}$, where $\mathbf{R-Alg}$\ denotes the
category of $\mathbf{R}$-algebras.
\end{remark}

\begin{remark}
Given Weil algebras $W_{1}$\ and $W_{2}$, we have%
\begin{equation}
\mathbf{y}W_{1}\times\mathbf{y}W_{2}\cong\mathbf{y}\left(  W_{1}\otimes
W_{2}\right)  \label{2.2}%
\end{equation}

\end{remark}

\begin{remark}
As is well known (cf. \S 8.7 of \cite{aw1}), given Weil spaces $F$\ and $G$,
their exponential $F^{G}$\ in $\mathbf{Weil}$\ is provided by%
\begin{equation}
\mathrm{Hom}_{\mathbf{Weil}}\left(  \mathbf{y}\_\times G,F\right)  \label{2.3}%
\end{equation}

\end{remark}

\begin{proposition}
\label{p2.2}For any Weil space $F$\ and any Weil algebra $W$, $F^{W}$ and
$F^{\mathbf{y}W}$\ are naturally isomorphic, namely,
\[
F^{W}\cong F^{\mathbf{y}W}%
\]
where the left-hand side stands for the Weil prolongation $F^{W}$\ of $F$\ by
$W$,\ while the right-hand side stands for the exponential $F^{\mathbf{y}W}%
$\ in the topos $\mathbf{Weil}$.
\end{proposition}

\begin{proof}
The proof is so simple as follows:%
\begin{align*}
&  F^{\mathbf{y}W}\\
&  =\mathrm{Hom}\left(  \mathbf{y}\_\times\mathbf{y}W,F\right) \\
&  \text{[(\ref{2.3})]}\\
&  \cong\mathrm{Hom}\left(  \mathbf{y}\left(  \_\otimes W\right)  ,F\right) \\
&  \text{[(\ref{2.2})]}\\
&  \cong F\left(  \_\otimes W\right) \\
&  \text{[(\ref{2.1})]}\\
&  =F^{W}%
\end{align*}

\end{proof}

\begin{corollary}
\label{cp2.2}Given a Weil algebra $W$\ together with Weil spaces $F$\ and $G$,
$\left(  F^{G}\right)  ^{W}$ and $\left(  F^{W}\right)  ^{G}$\ are naturally
isomorphic, namely,
\[
\left(  F^{G}\right)  ^{W}\cong\left(  F^{W}\right)  ^{G}%
\]

\end{corollary}

\begin{proof}
We have%
\begin{align*}
&  \left(  F^{G}\right)  ^{W}\\
&  \cong\left(  F^{G}\right)  ^{\mathbf{y}W}\\
&  \text{\lbrack by Proposition \ref{p2.2}]}\\
&  \cong\left(  F^{\mathbf{y}W}\right)  ^{G}\\
&  \cong\left(  F^{W}\right)  ^{G}\\
&  \text{\lbrack by Proposition \ref{p2.2}]}%
\end{align*}

\end{proof}

\begin{corollary}
For any Weil algebra $W$, the functor $\left(  \_\right)  ^{W}:\mathbf{Weil}%
\rightarrow\mathbf{Weil}$ preserves limits, particularly, products.
\end{corollary}

\begin{proof}
Since the functor $\left(  \_\right)  ^{W}$\ is of its left adjoint $\left(
\_\right)  \times\mathbf{y}W$ (cf. Proposition 8.13 of \cite{aw1}), the
desired result follows readily from the well known theorem claiming that a
functor being of its left adjoint preserves limits (cf. Proposition 9.14 of
\cite{aw1}).
\end{proof}

\begin{notation}
We denote by $\mathbb{R}$\ the forgetful functor $\mathfrak{W}\rightarrow
\mathbf{Sets}$, which is surely a Weil space. It can be defined also as%
\[
\mathbb{R}=\widetilde{\mathbf{y}}\left(  \mathbf{R}\left[  X\right]  \right)
\]

\end{notation}

\begin{remark}
The Weil space $\mathbb{R}$\ is canonically regarded as an $\mathbf{R}%
$-algebra object in the category $\mathbf{Weil}$.
\end{remark}

\begin{remark}
Since $\mathbb{R}$\ is an $\mathbf{R}$-algebra object in the category
$\mathbf{Weil}$, we can define, after \S I.16 of \cite{koc1}, another
$\mathbf{R}$-algebra object $\mathbb{R}\otimes W$\ in the category
$\mathbf{Weil}$\ for any Weil algebra $W$.
\end{remark}

\begin{notation}
We denote by $\mathbf{R-Alg}\left(  \mathbf{Weil}\right)  $ the category of
$\mathbf{R}$-algebra objects in the category $\mathbf{Weil}$.
\end{notation}

\begin{proposition}
\label{p2.3}The functors%
\[
\mathbb{R}^{\mathbf{y}\left(  \_\right)  },\mathbb{R}\otimes\left(  \_\right)
:\mathfrak{W}\rightarrow\mathbf{R-Alg}\left(  \mathbf{Weil}\right)
\]
are naturally isomorphic.
\end{proposition}

\begin{proof}
We have%
\begin{align*}
&  \mathbb{R}^{\mathbf{y}W}\left(  W^{\prime}\right) \\
&  \cong\mathbb{R}^{W}\left(  W^{\prime}\right) \\
&  \text{[By Proposition \ref{p2.2}]}\\
&  =W^{\prime}\otimes W
\end{align*}

\end{proof}

\section{\label{s3}Microlinearity}

Not all Weil spaces are susceptible to the techniques of classical
differential geometry, so that there should be a criterion by which we can
select decent ones.

\begin{definition}
A Weil space $F$\ is called \underline{microlinear} provided that a finite
limit diagram $\mathcal{D}$\ in $\mathfrak{W}$\ always yields a limit diagram
$F^{\mathcal{D}}$\ in $\mathbf{Weil}$.
\end{definition}

\begin{proposition}
\label{p3.1}We have the following:

\begin{enumerate}
\item The Weil space $\mathbb{R}$\ is microlinear.

\item The limit of a diagram of microlinear Weil spaces is microlinear.

\item Given Weil spaces $F$\ and $G$, if $F$\ is microlinear, then the
exponential $F^{G}$\ is also microlinear.
\end{enumerate}
\end{proposition}

\begin{proof}
The first statement follows from Proposition \ref{p2.3}. The second statement
follows from the well-known fact that double limits commute. The third
statement follows from Corollary \ref{cp2.2}.
\end{proof}

It is easy to see that

\begin{proposition}
A Weil space $F$\ is microlinear iff the diagram%
\[
F\left(  W\otimes\mathcal{D}\right)
\]
is a limit diagram for any Weil algebra $W$\ and any finite limit diagram
$\mathcal{D}$\ of Weil algebras.
\end{proposition}

\begin{proof}
By Proposition 8.7 of .\cite{aw1}
\end{proof}

\section{\label{s4}Weil Categories}

\begin{definition}
A \underline{Weil category} is a couple $\left(  \mathcal{K},\mathbf{D}%
\right)  $, where

\begin{enumerate}
\item $\mathcal{K}$ is a topos.

\item $\mathbf{D}:\widetilde{\mathfrak{W}}^{\mathrm{op}}\rightarrow
\mathcal{K}$ is a product-preserving functor. In particular, we have%
\[
\mathbf{D}\left(  \mathbf{R}\right)  =1
\]
where $1$\ denotes the terminal object in $\mathcal{K}$.
\end{enumerate}
\end{definition}

\begin{remark}
The entity $\mathbf{D}$ is called a \underline{Dubuc functor} with enthroning
his pioneering work in \cite{du1}.
\end{remark}

Now some examples are in order.

\begin{example}
The first example of a Weil category has already been discussed in
\S \ref{s2}, namely,%
\begin{align*}
\mathcal{K}  &  =\mathbf{Weil}\\
\mathbf{D}  &  =\widetilde{\mathbf{y}}%
\end{align*}
Indeed, this is the paradigm of our new concept of a Weil category, just as
the category $\mathbf{Sets}$\ is the paradigm of the prevailing concept of a topos.
\end{example}

\begin{notation}
We denote by $C^{\infty}\mathbf{-Alg}$\ the category of $C^{\infty}$-algebras.
\end{notation}

\begin{example}
Let $\mathbf{L}$\ be a class of $C^{\infty}$-algebras encompassing all
$C^{\infty}$-algebras of the form $C^{\infty}\left(  \mathbf{R}^{n}\right)
\otimes W$ with $W$ being a Weil algebra (cf. Theorem III.5.3 of
\cite{koc1}.). We define a functor $i_{\widetilde{\mathfrak{W}},C^{\infty
}\mathbf{-Alg}}:\widetilde{\mathfrak{W}}\rightarrow C^{\infty}\mathbf{-Alg}$
as%
\[
i_{\widetilde{\mathfrak{W}},C^{\infty}\mathbf{-Alg}}\left(  \mathbf{R}\left[
X_{1},...,X_{n}\right]  \otimes W\right)  =C^{\infty}\left(  \mathbf{R}%
^{n}\right)  \otimes W
\]
Putting down $\mathbf{L}$\ as a full subcategory of the category $C^{\infty
}\mathbf{-Alg}$, consider a subcanonical Grothendieck topology $J$ on the
category $\mathbf{L}^{\mathrm{op}}$. We let $\mathcal{K}$\ be the category of
all sheaves on the site $\left(  \mathbf{L}^{\mathrm{op}},J\right)  $. The
Dubuc functor $\mathbf{D}$\ is defined as%
\[
\mathbf{D}=\mathbf{y}\circ i_{\widetilde{\mathfrak{W}},C^{\infty}%
\mathbf{-Alg}}%
\]
where $\mathbf{y}$\ stands for the Yoneda embedding.
\end{example}

\begin{remark}
Such examples have been discussed amply in the context of well-adapted models
of synthetic differential geometry without being conscious of Weil categories
at all. The reader is referred to \cite{koc1} and \cite{mo1} for them.
\end{remark}

Now we fix a Weil category $\left(  \mathcal{K},\mathbf{D}\right)
$\ throughout the rest of this section. Weil functors are to be defined within
our framework of a Weil category.

\begin{definition}
The bifunctor $\mathbf{T}:\mathfrak{W}\times\mathcal{K}\rightarrow\mathcal{K}$
is defined to be%
\[
\mathbf{T}\left(  \left(  \_\right)  ,\left(  \cdot\right)  \right)
\cong\left(  \cdot\right)  ^{\mathbf{D}\left(  \_\right)  }%
\]

\end{definition}

We give some elementary properties with respect to $\mathbf{T}$.

\begin{proposition}
We have the following:

\begin{itemize}
\item The functor $\mathbf{T}\left(  \mathbf{R},\left(  \_\right)  \right)
$\ and the identity functor of $\mathcal{K}$, both of which are $\mathcal{K}%
\rightarrow\mathcal{K}$,\ are naturally isomorphic, namely,
\[
\mathbf{T}\left(  \mathbf{R},\left(  \_\right)  \right)  \cong\left(
\_\right)
\]

\item The trifunctors $\mathbf{T}\left(  \left(  \cdot_{2}\right)
,\mathbf{T}\left(  \left(  \cdot_{1}\right)  ,\left(  \_\right)  \right)
\right)  $ and $\mathbf{T}\left(  \left(  \cdot_{1}\right)  \otimes\left(
\cdot_{2}\right)  ,\left(  \_\right)  \right)  $, both of which are
$\mathfrak{W}\times\mathfrak{W}\times\mathcal{K}\rightarrow\mathcal{K}$, are
naturally isomorphic, namely,%
\[
\mathbf{T}\left(  \left(  \cdot_{2}\right)  ,\mathbf{T}\left(  \left(
\cdot_{1}\right)  ,\left(  \_\right)  \right)  \right)  \cong\mathbf{T}\left(
\left(  \cdot_{1}\right)  \otimes\left(  \cdot_{2}\right)  ,\left(  \_\right)
\right)
\]
for any Weil space $F$\ and any Weil algebras $W_{1}$ and $W_{2}$.
\end{itemize}
\end{proposition}

\begin{proposition}
Given a Weil algebra $W$, the functor $\mathbf{T}\left(  W,\cdot\right)
:\mathcal{K}\rightarrow\mathcal{K}$ preserves limits.
\end{proposition}

\begin{proof}
Since the functor $\mathbf{T}\left(  W,\cdot\right)  :\mathcal{K}%
\rightarrow\mathcal{K}$\ is of its left adjoint $\left(  \cdot\right)
\times\mathbf{D}W:\mathcal{K}\rightarrow\mathcal{K}$, the desired result
follows readily from the well known theorem claiming that a functor being of
its left adjoint preserves limits (cf. Proposition 9.14 of \cite{aw1}).
\end{proof}

\begin{proposition}
The trifunctors $\mathbf{T}\left(  \left(  \_\right)  ,\left(  \cdot
_{1}\right)  ^{\left(  \cdot_{2}\right)  }\right)  ,\mathbf{T}\left(  \left(
\_\right)  ,\left(  \cdot_{1}\right)  \right)  ^{\left(  \cdot_{2}\right)
}:\mathfrak{W}\times\mathcal{K}\times\mathcal{K}\rightarrow\mathcal{K}$ are
naturally isomorphic, namely,%
\[
\mathbf{T}\left(  \left(  \_\right)  ,\left(  \cdot_{1}\right)  ^{\left(
\cdot_{2}\right)  }\right)  \cong\mathbf{T}\left(  \left(  \_\right)  ,\left(
\cdot_{1}\right)  \right)  ^{\left(  \cdot_{2}\right)  }%
\]

\end{proposition}

\begin{proof}
We have%
\begin{align*}
&  \mathbf{T}\left(  \left(  \_\right)  ,\left(  \cdot_{1}\right)  ^{\left(
\cdot_{2}\right)  }\right) \\
&  =\left(  \left(  \cdot_{1}\right)  ^{\left(  \cdot_{2}\right)  }\right)
^{\mathbf{D}\left(  \_\right)  }\\
&  \cong\left(  \left(  \cdot_{1}\right)  ^{\mathbf{D}\left(  \_\right)
}\right)  ^{\left(  \cdot_{2}\right)  }\\
&  =\mathbf{T}\left(  \left(  \_\right)  ,\left(  \cdot_{1}\right)  \right)
^{\left(  \cdot_{2}\right)  }%
\end{align*}

\end{proof}

An $\mathbf{R}$-algebra\ object is to be introduced within our framework of a
Weil category.

\begin{notation}
The entity $\mathbf{D}\left(  \mathbf{R}\left[  X\right]  \right)  $ is
denoted by $\mathbb{R}$.
\end{notation}

It is in nearly every mathematician's palm to see that

\begin{proposition}
\label{p4.1}The entity $\mathbb{R}$ is a commutative $\mathbf{R}$-algebra
object in $\mathcal{K}$ with respect to the following addition,
multiplication, scalar multiplication by $\alpha\in\mathbf{R}$ and unity:%
\begin{align*}
\mathbf{D}\left(  X+Y\leftarrow X\right)   &  :\mathbb{R}\times\mathbb{R}%
=\mathbf{D}\left(  \mathbf{R}\left[  X,Y\right]  \right)  \rightarrow
\mathbf{D}\left(  \mathbf{R}\left[  X\right]  \right)  =\mathbb{R}\\
\mathbf{D}\left(  XY\leftarrow X\right)   &  :\mathbb{R}\times\mathbb{R}%
=\mathbf{D}\left(  \mathbf{R}\left[  X,Y\right]  \right)  \rightarrow
\mathbf{D}\left(  \mathbf{R}\left[  X\right]  \right)  =\mathbb{R}\\
\mathbf{D}\left(  \alpha X\leftarrow X\right)   &  :\mathbb{R}=\mathbf{D}%
\left(  \mathbf{R}\left[  X\right]  \right)  \rightarrow\mathbf{D}\left(
\mathbf{R}\left[  X\right]  \right)  =\mathbb{R}\\
\mathbf{D}\left(  1\leftarrow X\right)   &  :1=\mathbf{D}\left(
\mathbf{R}\right)  \rightarrow\mathbf{D}\left(  \mathbf{R}\left[  X\right]
\right)  =\mathbb{R}%
\end{align*}

\end{proposition}

\begin{notation}
The above four morphisms are denoted by%
\begin{align*}
+_{\mathbb{R}}  &  :\mathbb{R}\times\mathbb{R}\rightarrow\mathbb{R}\\
\cdot_{\mathbb{R}}  &  :\mathbb{R}\times\mathbb{R}\rightarrow\mathbb{R}\\
\alpha\cdot &  :\mathbb{R}\rightarrow\mathbb{R}\\
1_{\mathbb{R}}  &  :1\rightarrow\mathbb{R}%
\end{align*}
in order.
\end{notation}

\begin{notation}
The entity $\mathbf{D}\left(  \mathbf{R}\left[  X\right]  /\left(
X^{2}\right)  \right)  $ is denoted by $D$.
\end{notation}

\begin{proposition}
The $\mathbf{R}$-algebra object $\mathbb{R}$\ operates canonically on $D$ in
$\mathcal{K}$. To be specific, we have the following morphism:%
\begin{align*}
\mathbf{D}\left(  ZX/\left(  X^{2}\right)  \leftarrow X/\left(  X^{2}\right)
\right)   &  :\mathbb{R}\times D=\mathbf{D}\left(  \mathbf{R}\left[  Z\right]
\right)  \times\mathbf{D}\left(  \mathbf{R}\left[  X\right]  /\left(
X^{2}\right)  \right)  =\\
\mathbf{D}\left(  \mathbf{R}\left[  X,Z\right]  /\left(  X^{2}\right)
\right)   &  \rightarrow\mathbf{D}\left(  \mathbf{R}\left[  X\right]  /\left(
X^{2}\right)  \right)  =D
\end{align*}

\end{proposition}

\begin{notation}
The above morphism is denoted by $\cdot_{\mathbb{R},D}$.
\end{notation}

\begin{proposition}
It makes the following diagrams commutative:

\begin{enumerate}
\item
\[%
\begin{tabular}
[c]{lll}%
$\mathbb{R}\times\mathbb{R}\times D$ & $\longrightarrow$ & $\mathbb{R}\times
D$\\
& $\searrow$ & $\downarrow$\\
&  & $D$%
\end{tabular}
\ \
\]
:where the horizontal arrow is $+_{\mathbb{R}}\times D:\mathbb{R}%
\times\mathbb{R}\times D\rightarrow\mathbb{R}\times D$, the vertical arrow is
$\cdot_{\mathbb{R},D}:\mathbb{R}\times D\rightarrow D$, and the slant arrow is%
\begin{align*}
\mathbf{D}\left(  Z_{1}X+Z_{2}X/\left(  X^{2}\right)  \leftarrow X/\left(
X^{2}\right)  \right)   &  :\mathbb{R}\times\mathbb{R}\times D=\\
\mathbf{D}\left(  \mathbf{R}\left[  Z_{1}\right]  \right)  \times
\mathbf{D}\left(  \mathbf{R}\left[  Z_{2}\right]  \right)  \times
\mathbf{D}\left(  \mathbf{R}\left[  X\right]  /\left(  X^{2}\right)  \right)
&  =\mathbf{D}\left(  \mathbf{R}\left[  Z_{1},Z_{2},X\right]  /\left(
X^{2}\right)  \right)  \rightarrow\mathbf{D}\left(  \mathbf{R}\left[
X\right]  /\left(  X^{2}\right)  \right)  =D
\end{align*}

\item
\[%
\begin{tabular}
[c]{lll}%
$\mathbb{R}\times\mathbb{R}\times D$ & $\longrightarrow$ & $\mathbb{R}\times
D$\\
$\downarrow$ &  & $\downarrow$\\
$\mathbb{R}\times D$ & $\longrightarrow$ & $D$%
\end{tabular}
\]
where the upper horizontal arrow is $\cdot_{\mathbb{R}}\times D:\mathbb{R}%
\times\mathbb{R}\times D\rightarrow\mathbb{R}\times D$, the lower horizontal
arrow is $\cdot_{\mathbb{R},D}:$,$\mathbb{R}\times D\rightarrow D$ the left
vertical arrow is $\mathbb{R}\times\cdot_{\mathbb{R},D}:\mathbb{R}%
\times\mathbb{R}\times D\rightarrow\mathbb{R}\times D$, and the right vertical
arrow is $\cdot_{\mathbb{R},D}:\mathbb{R}\times D\rightarrow D$.

\item
\[%
\begin{tabular}
[c]{lll}%
$\mathbb{R}\times D$ & $\longrightarrow$ & $D$\\
$\uparrow$ & $\nearrow$ & \\
$1\times D=D$ &  &
\end{tabular}
\]
where the horizontal arrow is $\cdot_{\mathbb{R},D}:\mathbb{R}\times
D\rightarrow D$, the vertical arrow is $1_{\mathbb{R}}\times D:1\times
D\rightarrow\mathbb{R}\times D$, and the slant arrow is $\mathrm{id}%
_{D}:D\rightarrow D$.
\end{enumerate}
\end{proposition}

\begin{remark}
We have no canonical addition in $D$. In other words, we could not define
addition in $D$\ in such a way as%
\begin{align*}
\mathbf{D}\left(  \left(  X+Y\right)  /\left(  X^{2},Y^{2}\right)  \leftarrow
X/\left(  X^{2}\right)  \right)   &  :D\times D=\mathbf{R}\left[  X\right]
/\left(  X^{2}\right)  \times\mathbf{R}\left[  Y\right]  /\left(
Y^{2}\right)  =\\
\mathbf{R}\left[  X,Y\right]  /\left(  X^{2},Y^{2}\right)   &  \rightarrow
\mathbf{R}\left[  X\right]  /\left(  X^{2}\right)  =D
\end{align*}
This would simply be meaningless, because%
\[
\left(  X+Y\right)  /\left(  X^{2},Y^{2}\right)  \leftarrow X/\left(
X^{2}\right)
\]
is not well-defined.
\end{remark}

\begin{remark}
We have the canonical morphism $D\rightarrow\mathbb{R}$. Specifically
speaking, it is to be%
\[
\mathbf{D}\left(  X/\left(  X^{2}\right)  \leftarrow Z\right)  :D=\mathbf{D}%
\left(  \mathbf{R}\left[  X\right]  /\left(  X^{2}\right)  \right)
\rightarrow\mathbf{D}\left(  \mathbf{R}\left[  Z\right]  \right)  =\mathbb{R}%
\]

\end{remark}

Many significant concepts and theorems of topos theory can quite easily be
transferred into the theory of Weil categories surely with due modifications.
In particular, we have

\begin{theorem}
\label{t4.1}(\underline{The Fundamental Theorem for Weil Categories}, cf.
Theorem 4.19 in \cite{bel1} and Theorem 1 in \S IV.7 of \cite{ma1}) Let
$\left(  \mathcal{K},\mathbf{D}\right)  $\ be a Weil category with
$M\in\mathcal{K}$. Then the slice category $\mathcal{K}/M$\ endowed with a
Dubuc functor $\mathbf{D}_{M}:\rightarrow\mathcal{K}/M$ is a Weil category, where

\begin{itemize}
\item $\mathbf{D}_{M}\left(  A\right)  $ is the canonical projection
$\mathbf{D}\left(  A\right)  \times M\rightarrow M$ for any $A\in
\widetilde{\mathfrak{W}}$, and

\item $\mathbf{D}_{M}\left(  f\right)  $ is $f\times M$ for any morphism
$f$\ in $\widetilde{\mathfrak{W}}$.
\end{itemize}
\end{theorem}

\begin{remark}
This theorem corresponds to so-called \underline{fiberwise differential
geometry}. In other words, the theorem claims that we can do differential
geometry \underline{fiberwise}.
\end{remark}

\section{\label{s5}Axiomatic Differential Geometry}

We fix a Weil category $\left(  \mathcal{K},\mathbf{D}\right)  $\ throughout
this section.

\begin{notation}
We introduce the following aliases:

\begin{itemize}
\item The entity $\mathbf{D}\left(  \mathbf{R}\left[  X,Y\right]  /\left(
X^{2},Y^{2},XY\right)  \right)  $ is denoted by $D\left(  2\right)  $.

\item The entity $\mathbf{D}\left(  \mathbf{R}\left[  X,Y,Z\right]  /\left(
X^{2},Y^{2},Z^{2},XY,XZ,YZ\right)  \right)  $ is denoted by $D\left(
3\right)  $.
\end{itemize}
\end{notation}

As a corollary of Proposition \ref{p4.1} and Theorem \ref{t4.1}, we have

\begin{proposition}
The canonical projection $\mathbb{R}\times M\rightarrow M$\ is a commutative
$\mathbf{R}$-algebra object in the slice category $\mathcal{K}/M$.
\end{proposition}

\begin{definition}
An object $M$\ in $\mathcal{K}$\ is called \underline{microlinear} provided
that a finite limit diagram $\mathcal{D}$\ in $\mathfrak{W}$\ always yields a
limit diagram $\mathbf{T}\left(  \mathcal{D},M\right)  $\ in $\mathcal{K}$.
\end{definition}

As in Proposition \ref{p3.1}, we have

\begin{proposition}
We have the following:

\begin{enumerate}
\item The limit of a diagram of microlinear objects in $\mathcal{K}$\ is microlinear.

\item Given objects $M$\ and $N$ in $\mathcal{K}$, if $M$\ is microlinear,
then the exponential $M^{N}$\ is also microlinear.
\end{enumerate}
\end{proposition}

\begin{theorem}
\label{t5.1}Let $M$\ be a microlinear object in $\mathcal{K}$. The entity
$M^{\mathbf{D}\left(  \mathbf{R}\rightarrow\mathbf{R}\left[  X\right]
/(X^{2})\right)  }:M^{D}=M^{\mathbf{D}\left(  \mathbf{R}\left[  X\right]
/(X^{2})\right)  }\rightarrow M^{\mathbf{D}\left(  \mathbf{R})\right)  }=M$ is
a $\left(  \mathbb{R}\times M\rightarrow M\right)  $-module object in the
slice category $\mathcal{K}/M$ with respect to the following addition and
scalar multiplication:

\begin{itemize}
\item The following diagram%
\[%
\begin{array}
[c]{ccc}%
\mathbf{R}\left[  X,Y\right]  /\left(  X^{2},Y^{2},XY\right)  & \rightarrow &
\mathbf{R}\left[  Y\right]  /\left(  Y^{2}\right) \\
\downarrow &  & \downarrow\\
\mathbf{R}\left[  X\right]  /\left(  X^{2}\right)  & \rightarrow & \mathbf{R}%
\end{array}
\]
is a pullback, where the upper horizontal arrow is%
\[
\left(  X,Y\right)  /\left(  X^{2},Y^{2},XY\right)  \rightarrow\left(
0,Y\right)  /\left(  Y^{2}\right)
\]
the lower horizontal arrow is%
\[
X/\left(  X^{2}\right)  \rightarrow0
\]
the left vertical arrow is%
\[
\left(  X,Y\right)  /\left(  X^{2},Y^{2},XY\right)  \rightarrow\left(
X,0\right)  /\left(  X^{2}\right)
\]
and the right vertical arrow is%
\[
Y/\left(  Y^{2}\right)  \rightarrow0
\]
Since $M$\ is microlinear, the diagram%
\[%
\begin{array}
[c]{ccc}%
M^{D\left(  2\right)  }=M^{\mathbf{D}\left(  \mathbf{R}\left[  X,Y\right]
/\left(  X^{2},Y^{2},XY\right)  \right)  } & \rightarrow & M^{\mathbf{D}%
\left(  \mathbf{R}\left[  Y\right]  /\left(  Y^{2}\right)  \right)  }=M^{D}\\
\downarrow &  & \downarrow\\
M^{D}=M^{\mathbf{D}\left(  \mathbf{R}\left[  X\right]  /\left(  X^{2}\right)
\right)  } & \rightarrow & M^{\mathbf{D}\left(  \mathbf{R}\right)  }=M
\end{array}
\]
is a pullback, where the upper horizontal arrow is%
\[
M^{\mathbf{D}\left(  \left(  X,Y\right)  /\left(  X^{2},Y^{2},XY\right)
\rightarrow\left(  0,Y\right)  /\left(  Y^{2}\right)  \right)  }%
\]
the lower horizontal arrow is%
\[
M^{\mathbf{D}\left(  X/\left(  X^{2}\right)  \rightarrow0\right)  }%
\]
the left vertical arrow is%
\[
M^{\mathbf{D}\left(  \left(  X,Y\right)  /\left(  X^{2},Y^{2},XY\right)
\rightarrow\left(  X,0\right)  /\left(  X^{2}\right)  \right)  }%
\]
and the right vertical arrow is%
\[
M^{\mathbf{D}\left(  Y/\left(  Y^{2}\right)  \rightarrow0\right)  }%
\]
Therefore we have%
\[
M^{D\left(  2\right)  }=M^{D}\times_{M}M^{D}%
\]
The morphism%
\begin{align*}
M^{\mathbf{D}\left(  \left(  X,Y\right)  /\left(  X^{2},Y^{2},XY\right)
\rightarrow\left(  X,X\right)  /\left(  X^{2}\right)  \right)  }  &
:M^{D}\times_{M}M^{D}=M^{D\left(  2\right)  }=M^{\mathbf{D}\left(
\mathbf{R}\left[  X,Y\right]  /\left(  X^{2},Y^{2},XY\right)  \right)  }\\
&  \rightarrow M^{\mathbf{D}\left(  \mathbf{R}\left[  X\right]  /\left(
X^{2}\right)  \right)  }=M^{D}%
\end{align*}
stands for addition and is denoted by $\varphi$.

\item The composition of the morphism%
\begin{align*}
\mathbf{D}\left(  XY/\left(  X^{2}\right)  \leftarrow X/\left(  X^{2}\right)
\right)  \times M^{D}  &  :\\
D\times\mathbb{R}\times M^{D}  &  =\mathbf{D}\left(  \mathbf{R}\left[
X\right]  /\left(  X^{2}\right)  \right)  \times\mathbf{D}\left(
\mathbf{R}\left[  Y\right]  \right)  \times M^{D}\\
&  \rightarrow\\
\mathbf{D}\left(  \mathbf{R}\left[  X\right]  /\left(  X^{2}\right)  \right)
\times M^{D}  &  =D\times M^{D}%
\end{align*}
and the evaluation morphism%
\[
D\times M^{D}\rightarrow M
\]
is denoted by $\widehat{\psi}_{1}:D\times\mathbb{R}\times M^{D}\rightarrow M$.
Its transpose $\psi_{1}:\mathbb{R}\times M^{D}\rightarrow M^{D}$\ stands for
scalar multiplication.
\end{itemize}
\end{theorem}

\begin{proof}
Here we deal only with the associativity of addition and the distibutivity of
scalar multiplication over addition, leaving verification of the other
rquisites of $M^{\mathbf{D}\left(  \mathbf{R}\rightarrow\mathbf{R}\left[
X\right]  /(X^{2})\right)  }:M^{D}=M^{\mathbf{D}\left(  \mathbf{R}\left[
X\right]  /(X^{2})\right)  }\rightarrow M^{\mathbf{D}\left(  \mathbf{R}%
)\right)  }=M$\ being a $\left(  \mathbb{R}\times M\rightarrow M\right)
$-module object in the category $\mathcal{K}/M$\ to the reader.

\begin{itemize}
\item The diagram%
\[%
\begin{array}
[c]{ccccc}
&  & \mathbf{R}\left[  X,Y,Z\right]  /\left(  X^{2},Y^{2},Z^{2}%
,XY,XZ,YZ\right)  &  & \\
& \swarrow & \downarrow & \searrow & \\
\mathbf{R}\left[  X\right]  /\left(  X^{2}\right)  &  & \mathbf{R}\left[
X\right]  /\left(  X^{2}\right)  &  & \mathbf{R}\left[  X\right]  /\left(
X^{2}\right) \\
& \searrow & \downarrow & \swarrow & \\
&  & \mathbf{R} &  &
\end{array}
\]
is a limit diagram, where the upper three arrows are%
\begin{align*}
\left(  X,Y,Z\right)  /\left(  X^{2},Y^{2},Z^{2},XY,XZ,YZ\right)   &
\rightarrow\left(  X,0,0\right)  /\left(  X^{2}\right) \\
\left(  X,Y,Z\right)  /\left(  X^{2},Y^{2},Z^{2},XY,XZ,YZ\right)   &
\rightarrow\left(  0,X,0\right)  /\left(  X^{2}\right) \\
\left(  X,Y,Z\right)  /\left(  X^{2},Y^{2},Z^{2},XY,XZ,YZ\right)   &
\rightarrow\left(  0,0,X\right)  /\left(  X^{2}\right)
\end{align*}
from left to right, and the lower three arrows are the same%
\[
X/\left(  X^{2}\right)  \rightarrow0
\]
Since $M$\ is microlinear, the diagram%
\[%
\begin{array}
[c]{ccccc}
&  &
\begin{array}
[c]{c}%
M^{D\left(  3\right)  }=\\
M^{\mathbf{D}\left(  \mathbf{R}\left[  X,Y,Z\right]  /\left(  X^{2}%
,Y^{2},Z^{2},XY,XZ,YZ\right)  \right)  }%
\end{array}
&  & \\
& \swarrow & \downarrow & \searrow & \\
M^{D}=M^{\mathbf{D}\left(  \mathbf{R}\left[  X\right]  /\left(  X^{2}\right)
\right)  } &  & M^{D}=M^{\mathbf{D}\left(  \mathbf{R}\left[  X\right]
/\left(  X^{2}\right)  \right)  } &  & M^{D}=M^{\mathbf{D}\left(
\mathbf{R}\left[  X\right]  /\left(  X^{2}\right)  \right)  }\\
& \searrow & \downarrow & \swarrow & \\
&  & M=M^{\mathbf{D}\left(  \mathbf{R}\right)  } &  &
\end{array}
\]
is a limit diagram, where the upper three arrows are%
\begin{align*}
&  M^{\mathbf{D}\left(  \left(  X,Y,Z\right)  /\left(  X^{2},Y^{2}%
,Z^{2},XY,XZ,YZ\right)  \rightarrow\left(  X,0,0\right)  /\left(
X^{2}\right)  \right)  }\\
&  M^{\mathbf{D}\left(  \left(  X,Y,Z\right)  /\left(  X^{2},Y^{2}%
,Z^{2},XY,XZ,YZ\right)  \rightarrow\left(  0,X,0\right)  /\left(
X^{2}\right)  \right)  }\\
&  M^{\mathbf{D}\left(  \left(  X,Y,Z\right)  /\left(  X^{2},Y^{2}%
,Z^{2},XY,XZ,YZ\right)  \rightarrow\left(  0,0,X\right)  /\left(
X^{2}\right)  \right)  }%
\end{align*}
from left to right, and the lower three arrows are the same%
\[
M^{\mathbf{D}\left(  X/\left(  X^{2}\right)  \rightarrow0\right)  }%
\]
Therefore we have%
\[
M^{D\left(  3\right)  }=M^{D}\times_{M}M^{D}\times_{M}M^{D}%
\]
It is now easy to see that the diagram%
\[%
\begin{array}
[c]{ccc}%
\begin{array}
[c]{c}%
M^{D}\times_{M}M^{D}\times_{M}M^{D}=M^{D\left(  3\right)  }=\\
M^{\mathbf{D}\left(  \mathbf{R}\left[  X,Y\right]  /\left(  X^{2},Y^{2}%
,Z^{2},XY,XZ,YZ\right)  \right)  }%
\end{array}
& \rightarrow &
\begin{array}
[c]{c}%
M^{\mathbf{D}\left(  \mathbf{R}\left[  X,Y\right]  /\left(  X^{2}%
,Y^{2},XY\right)  \right)  }=M^{D\left(  2\right)  }\\
=M^{D}\times_{M}M^{D}%
\end{array}
\\
\downarrow &  & \downarrow\\%
\begin{array}
[c]{c}%
M^{D}\times_{M}M^{D}=M^{D\left(  2\right)  }\\
=M^{\mathbf{D}\left(  \mathbf{R}\left[  X,Y\right]  /\left(  X^{2}%
,Y^{2},XY\right)  \right)  }%
\end{array}
& \rightarrow & M^{\mathbf{D}\left(  \mathbf{R}\left[  X\right]  /\left(
X^{2}\right)  \right)  }=M^{D}%
\end{array}
\]
is commutative, where the upper horizontal arrow is%
\[
M^{\mathbf{D}\left(  \left(  X,Y,Z\right)  /\left(  X^{2},Y^{2},Z^{2}%
,XY,XZ,YZ\right)  \rightarrow\left(  X,X,Y\right)  /\left(  X^{2}%
,Y^{2},XY\right)  \right)  }%
\]
the lower horizontal arrow is%
\[
M^{\mathbf{D}\left(  \left(  X,Y\right)  /\left(  X^{2},Y^{2},XY\right)
\rightarrow\left(  X,X\right)  /\left(  X^{2}\right)  \right)  }%
\]
the left vertical arrow is%
\[
M^{\mathbf{D}\left(  \left(  X,Y,Z\right)  /\left(  X^{2},Y^{2},Z^{2}%
,XY,XZ,YZ\right)  \rightarrow\left(  X,Y,Y\right)  /\left(  X^{2}%
,Y^{2},XY\right)  \right)  }%
\]
and the right vertical arrow is%
\[
M^{\mathbf{D}\left(  \left(  X,Y\right)  /\left(  X^{2},Y^{2},XY\right)
\rightarrow\left(  X,X\right)  /\left(  X^{2}\right)  \right)  }%
\]
We have just established the associativity of addition.

\item The proof of the distibutivity of scalar multiplication over
addition\ is divided into three steps:

\begin{enumerate}
\item The composition of the morphism%
\begin{align*}
\mathbf{D}\left(  \left(  XZ,YZ\right)  /\left(  X^{2},Y^{2},XY\right)
\leftarrow\left(  X,Y\right)  /\left(  X^{2},Y^{2},XY\right)  \right)  \times
M^{D\left(  2\right)  } &  :\\
D\left(  2\right)  \times\mathbb{R}\times M^{D\left(  2\right)  } &  =\\
\mathbf{D}\left(  \mathbf{R}\left[  X,Y\right]  /\left(  X^{2},Y^{2}%
,XY\right)  \right)  \times\mathbf{D}\left(  \mathbf{R}\left[  Z\right]
\right)  \times M^{D\left(  2\right)  } &  \rightarrow\\
\mathbf{D}\left(  \mathbf{R}\left[  X,Y\right]  /\left(  X^{2},Y^{2}%
,XY\right)  \right)  \times M^{D\left(  2\right)  } &  =D\left(  2\right)
\times M^{D\left(  2\right)  }%
\end{align*}
and the evaluation morphism%
\[
D\left(  2\right)  \times M^{D\left(  2\right)  }\rightarrow M
\]
is denoted by $\widehat{\psi}_{2}:D\left(  2\right)  \times\mathbb{R}\times
M^{D\left(  2\right)  }\rightarrow M$. Its transpose is denoted by $\psi
_{2}:\mathbb{R}\times M^{D\left(  2\right)  }\rightarrow M^{D\left(  2\right)
}$. And the composition of the morphism%
\begin{align*}
&  \mathbf{D}\left(  \left(  XZ_{1},YZ_{2}\right)  /\left(  X^{2}%
,Y^{2},XY\right)  \leftarrow\left(  X,Y\right)  /\left(  X^{2},Y^{2}%
,XY\right)  \right)  \times M^{D\left(  2\right)  }:D\left(  2\right)
\times\mathbb{R}\times\mathbb{R}\times M^{D\left(  2\right)  }=\\
&  \mathbf{D}\left(  \mathbf{R}\left[  X,Y\right]  /\left(  X^{2}%
,Y^{2},XY\right)  \right)  \times\mathbf{D}\left(  \mathbf{R}\left[
Z_{1},Z_{2}\right]  \right)  \times M^{D\left(  2\right)  }\rightarrow
\mathbf{D}\left(  \mathbf{R}\left[  X,Y\right]  /\left(  X^{2},Y^{2}%
,XY\right)  \right)  \times M^{D\left(  2\right)  }=\\
&  D\left(  2\right)  \times M^{D\left(  2\right)  }%
\end{align*}
and the evaluation morphism%
\[
D\left(  2\right)  \times M^{D\left(  2\right)  }\rightarrow M
\]
is denoted by $\widehat{\chi}:D\left(  2\right)  \times\mathbb{R}%
\times\mathbb{R}\times M^{D\left(  2\right)  }\rightarrow M$. Its transpose is
denoted by $\chi:\mathbb{R}\times\mathbb{R}\times M^{D\left(  2\right)
}\rightarrow M^{D\left(  2\right)  }$. It is easy to see that the diagram%
\[%
\begin{tabular}
[c]{lll}%
$\mathbb{R}\times M^{D\left(  2\right)  }$ &  & \\
$\downarrow$ & $\searrow$ & \\
$\mathbb{R}\times\mathbb{R}\times M^{D\left(  2\right)  }$ & $\longrightarrow$
& $M^{D\left(  2\right)  }$%
\end{tabular}
\ \ \ \ \ \
\]
commutes, where the vertical arrow is%
\begin{align*}
\mathbf{D}\left(  \left(  Z,Z\right)  \leftarrow\left(  Z_{1},Z_{2}\right)
\right)  \times M^{D\left(  2\right)  } &  :\mathbb{R}\times M^{D\left(
2\right)  }=\mathbf{D}\left(  \mathbf{R}\left[  Z\right]  \right)  \times
M^{D\left(  2\right)  }\rightarrow\\
\mathbf{D}\left(  \mathbf{R}\left[  Z_{1},Z_{2}\right]  \right)  \times
M^{D\left(  2\right)  } &  =\mathbb{R}\times\mathbb{R}\times M^{D\left(
2\right)  }%
\end{align*}
the horizontal arrow is%
\[
\chi:\mathbb{R}\times\mathbb{R}\times M^{D\left(  2\right)  }\rightarrow
M^{D\left(  2\right)  }%
\]
and the slant arrow is%
\[
\psi_{2}:\mathbb{R}\times M^{D\left(  2\right)  }\rightarrow M^{D\left(
2\right)  }%
\]
It is also easy to see that the morphism $\chi:\mathbb{R}\times\mathbb{R}%
\times M^{D\left(  2\right)  }\rightarrow M^{D\left(  2\right)  }$\ can be
defined to be%
\begin{align*}
\psi_{2}\times_{M}\psi_{2} &  :\mathbb{R}\times\mathbb{R}\times M^{D\left(
2\right)  }=\mathbb{R}\times\mathbb{R}\times\left(  M^{D}\times_{M}%
M^{D}\right)  =\\
\left(  \mathbb{R}\times M^{D}\right)  \times_{M}\left(  \mathbb{R}\times
M^{D}\right)   &  \rightarrow M^{D}\times_{M}M^{D}=M^{D\left(  2\right)  }%
\end{align*}

\item Let us consider the following diagram:%
\begin{equation}%
\begin{array}
[c]{ccccc}%
D\times\mathbb{R}\times M^{D} & \longleftarrow & D\times\mathbb{R}\times
M^{D\left(  2\right)  } & \longrightarrow & D\left(  2\right)  \times
\mathbb{R}\times M^{D\left(  2\right)  }\\
\downarrow & \fbox{1} & \downarrow & \fbox{2} & \downarrow\\
D\times M^{D} & \longleftarrow & D\times M^{D\left(  2\right)  } &
\longrightarrow & D\left(  2\right)  \times M^{D\left(  2\right)  }\\
& \searrow & \fbox{3} & \swarrow & \\
&  & M &  &
\end{array}
\label{t5.1.1}%
\end{equation}
where the upper two horizontal arrows are%
\begin{align*}
D\times\mathbb{R}\times\varphi &  :D\times\mathbb{R}\times M^{D\left(
2\right)  }\rightarrow D\times\mathbb{R}\times M^{D}\\
\mathbf{D}\left(  \left(  X,X\right)  /\left(  X^{2}\right)  \leftarrow\left(
X,Y\right)  /\left(  X^{2},Y^{2},XY\right)  \right)  \times\mathbb{R}\times
M^{D\left(  2\right)  }  &  :\\
D\times\mathbb{R}\times M^{D\left(  2\right)  }  &  =\mathbf{D}\left(
\mathbf{R}\left[  X\right]  /\left(  X^{2}\right)  \right)  \times
\mathbb{R}\times M^{D\left(  2\right)  }\\
&  \rightarrow\\
\mathbf{D}\left(  \mathbf{R}\left[  X,Y\right]  /\left(  X^{2},Y^{2}%
,XY\right)  \right)  \times\mathbb{R}\times M^{D\left(  2\right)  }  &
=D\left(  2\right)  \times\mathbb{R}\times M^{D\left(  2\right)  }%
\end{align*}
from left to right, the lower two horizontal arrow are%
\begin{align*}
D\times\varphi &  :D\times M^{D\left(  2\right)  }\rightarrow D\times M^{D}\\
\mathbf{D}\left(
\begin{array}
[c]{c}%
\left(  X,X\right)  /\left(  X^{2}\right)  \leftarrow\\
\left(  X,Y\right)  /\left(  X^{2},Y^{2},XY\right)
\end{array}
\right)  \times M^{D\left(  2\right)  }  &  :D\times M^{D\left(  2\right)
}=\mathbf{D}\left(  \mathbf{R}\left[  X\right]  /\left(  X^{2}\right)
\right)  \times M^{D\left(  2\right)  }\rightarrow\\
\mathbf{D}\left(  \mathbf{R}\left[  X,Y\right]  /\left(  X^{2},Y^{2}%
,XY\right)  \right)  \times M^{D\left(  2\right)  }  &  =D\left(  2\right)
\times M^{D\left(  2\right)  }%
\end{align*}
from left to right, the three vertical arrows are%
\begin{align*}
\mathbf{D}\left(  XY/\left(  X^{2}\right)  \leftarrow X/\left(  X^{2}\right)
\right)  \times M^{D}  &  :\\
D\times\mathbb{R}\times M^{D}  &  =\mathbf{D}\left(  \mathbf{R}\left[
X\right]  /\left(  X^{2}\right)  \right)  \times\mathbf{D}\left(
\mathbf{R}\left[  Y\right]  \right)  \times M^{D}\\
&  \rightarrow\mathbf{D}\left(  \mathbf{R}\left[  X\right]  /\left(
X^{2}\right)  \right)  \times M^{D}=D\times M^{D}\\
\mathbf{D}\left(  XY/\left(  X^{2}\right)  \leftarrow X/\left(  X^{2}\right)
\right)  \times M^{D\left(  2\right)  }  &  :\\
D\times\mathbb{R}\times M^{D\left(  2\right)  }  &  =\mathbf{D}\left(
\mathbf{R}\left[  X\right]  /\left(  X^{2}\right)  \right)  \times
\mathbf{D}\left(  \mathbf{R}\left[  Y\right]  \right)  \times M^{D\left(
2\right)  }\\
&  \rightarrow\\
\mathbf{D}\left(  \mathbf{R}\left[  X\right]  /\left(  X^{2}\right)  \right)
\times M^{D\left(  2\right)  }  &  =D\times M^{D\left(  2\right)  }\\
\mathbf{D}\left(
\begin{array}
[c]{c}%
\left(  XZ,YZ\right)  /\left(  X^{2},Y^{2},XY\right) \\
\leftarrow\left(  X,Y\right)  /\left(  X^{2},Y^{2},XY\right)
\end{array}
\right)  \times M^{D\left(  2\right)  }  &  :\\
D\left(  2\right)  \times\mathbb{R}\times M^{D\left(  2\right)  }  &  =\\
\mathbf{D}\left(  \mathbf{R}\left[  X,Y\right]  /\left(  X^{2},Y^{2}%
,XY\right)  \right)  \times\mathbf{D}\left(  \mathbf{R}\left[  Z\right]
\right)  \times M^{D\left(  2\right)  }  &  \rightarrow\\
\mathbf{D}\left(  \mathbf{R}\left[  X,Y\right]  /\left(  X^{2},Y^{2}%
,XY\right)  \right)  \times M^{D\left(  2\right)  }  &  =D\left(  2\right)
\times M^{D\left(  2\right)  }%
\end{align*}
from left to right, and the two slant arrows are the evaluation morphisms
$D\times M^{D}\rightarrow M$\ and $D\left(  2\right)  \times M^{D\left(
2\right)  }\rightarrow M$. In order to establish the commutativity of the
diagram (\ref{t5.1.1}), we will be engaged in the commutativity of the three
subdiagrams $\fbox{1}$, $\fbox{2}$\ and $\fbox{3}$\ in order. It is easy to
see that both the diagram $\fbox{1}$ and the digaram $\fbox{2}$\ commute. The
commutativity of the diagram $\fbox{1}$\ is a simple consequence of the fact
that $\left(  \_\right)  \times\left(  \_\right)  $\ is a bifunctor, while the
commutativity of the diagram $\fbox{2}$\ follows directly from that of the
following diagram%
\[%
\begin{tabular}
[c]{lll}%
$D\times\mathbb{R}$ & $\longrightarrow$ & $D\left(  2\right)  \times
\mathbb{R}$\\
$\downarrow$ &  & $\downarrow$\\
$D$ & $\longrightarrow$ & $D\left(  2\right)  $%
\end{tabular}
\ \ \ \ \ \ \ \ \ \ \ \ \ \ \
\]
where the two horizontal arrows are%
\begin{align*}
\mathbf{D}\left(  \left(  X,X\right)  /\left(  X^{2}\right)  \leftarrow\left(
X,Y\right)  /\left(  X^{2},Y^{2},XY\right)  \right)  \times\mathbb{R}  &  :\\
D\times\mathbb{R}  &  =\mathbf{D}\left(  \mathbf{R}\left[  X\right]  /\left(
X^{2}\right)  \right)  \times\mathbb{R}\\
&  \rightarrow\\
\mathbf{D}\left(  \mathbf{R}\left[  X,Y\right]  /\left(  X^{2},Y^{2}%
,XY\right)  \right)  \times\mathbb{R}  &  =D\left(  2\right)  \times
\mathbb{R}\\
\mathbf{D}\left(  \left(  X,X\right)  /\left(  X^{2}\right)  \leftarrow\left(
X,Y\right)  /\left(  X^{2},Y^{2},XY\right)  \right)   &  :D=\mathbf{D}\left(
\mathbf{R}\left[  X\right]  /\left(  X^{2}\right)  \right)  \rightarrow\\
\mathbf{D}\left(  \mathbf{R}\left[  X,Y\right]  /\left(  X^{2},Y^{2}%
,XY\right)  \right)   &  =D\left(  2\right)
\end{align*}
from top to bottom, and the two vertical arrows are%
\begin{align*}
\mathbf{D}\left(  XY/\left(  X^{2}\right)  \leftarrow X/\left(  X^{2}\right)
\right)   &  :\\
D\times\mathbb{R}  &  =\mathbf{D}\left(  \mathbf{R}\left[  X\right]  /\left(
X^{2}\right)  \right)  \times\mathbf{D}\left(  \mathbf{R}\left[  Y\right]
\right) \\
&  \rightarrow\mathbf{D}\left(  \mathbf{R}\left[  X\right]  /\left(
X^{2}\right)  \right)  =D\\
\mathbf{D}\left(
\begin{array}
[c]{c}%
\left(  XZ,YZ\right)  /\left(  X^{2},Y^{2},XY\right)  \leftarrow\\
\left(  X,Y\right)  /\left(  X^{2},Y^{2},XY\right)
\end{array}
\right)   &  :\\
D\left(  2\right)  \times\mathbb{R}  &  =\mathbf{D}\left(  \mathbf{R}\left[
X,Y\right]  /\left(  X^{2},Y^{2},XY\right)  \right)  \times\mathbf{D}\left(
\mathbf{R}\left[  Z\right]  \right) \\
&  \rightarrow\mathbf{D}\left(  \mathbf{R}\left[  X,Y\right]  /\left(
X^{2},Y^{2},XY\right)  \right)  =D\left(  2\right)
\end{align*}
from left to right. The commutativity of the diagram $\fbox{3}$\ follows from
the following commutative diagram of so-called parametrized adjunction (cf.
Theorem 3 in \S IV.7 of \cite{ma2}):%
\begin{equation}%
\begin{tabular}
[c]{lll}%
$\mathrm{Hom}_{\mathcal{K}}\left(  D\left(  2\right)  \times M^{D\left(
2\right)  },M\right)  $ & $\cong$ & $\mathrm{Hom}_{\mathcal{K}}\left(
M^{D\left(  2\right)  },M^{D\left(  2\right)  }\right)  $\\
$\downarrow$ & $\circlearrowleft$ & $\downarrow$\\
$\mathrm{Hom}_{\mathcal{K}}\left(  D\times M^{D\left(  2\right)  },M\right)  $
& $\cong$ & $\mathrm{Hom}_{\mathcal{K}}\left(  M^{D\left(  2\right)  }%
,M^{D}\right)  $\\
$\uparrow$ & $\circlearrowleft$ & $\uparrow$\\
$\mathrm{Hom}_{\mathcal{K}}\left(  D\times M^{D},M\right)  $ & $\cong$ &
$\mathrm{Hom}_{\mathcal{K}}\left(  M^{D},M^{D}\right)  $%
\end{tabular}
\ \ \ \ \ \ \ \ \ \ \ \ \label{t5.1.2}%
\end{equation}
where the left two vertical arrows are%
\begin{align*}
\mathrm{Hom}_{\mathcal{K}}\left(  \left(
\begin{array}
[c]{c}%
\mathbf{D}\left(
\begin{array}
[c]{c}%
\left(  X,X\right)  /\left(  X^{2}\right)  \leftarrow\\
\left(  X,Y\right)  /\left(  X^{2},Y^{2},XY\right)
\end{array}
\right)  :\\
D=\mathbf{D}\left(  \mathbf{R}\left[  X\right]  /\left(  X^{2}\right)
\right)  \rightarrow\\
\mathbf{D}\left(  \mathbf{R}\left[  X,Y\right]  /\left(  X^{2},Y^{2}%
,XY\right)  \right)  =D\left(  2\right)
\end{array}
\right)  \times M^{D\left(  2\right)  },M\right)   &  :\\
\mathrm{Hom}_{\mathcal{K}}\left(  D\left(  2\right)  \times M^{D\left(
2\right)  },M\right)   &  \rightarrow\mathrm{Hom}_{\mathcal{K}}\left(  D\times
M^{D\left(  2\right)  },M\right) \\
\mathrm{Hom}_{\mathcal{K}}\left(  D\times\varphi,M\right)   &  :\\
\mathrm{Hom}_{\mathcal{K}}\left(  D\times M^{D},M\right)   &  \rightarrow
\mathrm{Hom}_{\mathcal{K}}\left(  D\times M^{D\left(  2\right)  },M\right)
\end{align*}
from top to bottom, while the right vertical arrows are%
\begin{align*}
\mathrm{Hom}_{\mathcal{K}}\left(  M^{D\left(  2\right)  },\varphi\right)   &
:\mathrm{Hom}_{\mathcal{K}}\left(  M^{D\left(  2\right)  },M^{D\left(
2\right)  }\right)  \rightarrow\mathrm{Hom}_{\mathcal{K}}\left(  M^{D\left(
2\right)  },M^{D}\right) \\
\mathrm{Hom}_{\mathcal{K}}\left(  \varphi,M^{D}\right)   &  :\mathrm{Hom}%
_{\mathcal{K}}\left(  M^{D},M^{D}\right)  \rightarrow\mathrm{Hom}%
_{\mathcal{K}}\left(  M^{D\left(  2\right)  },M^{D}\right)
\end{align*}
from top to bottom. Choose%
\begin{align*}
\mathrm{id}_{M^{D\left(  2\right)  }}  &  \in\mathrm{Hom}_{\mathcal{K}}\left(
M^{D\left(  2\right)  },M^{D\left(  2\right)  }\right) \\
\mathrm{id}_{M^{D}}  &  \in\mathrm{Hom}_{\mathcal{K}}\left(  M^{D}%
,M^{D}\right)
\end{align*}
on the right of the diagram.(\ref{t5.1.2}). Then both yield the same morphism
in $\mathrm{Hom}_{\mathcal{K}}\left(  M^{D\left(  2\right)  },M^{D}\right)
$\ by application of their adjacent vertical arrows. The corresponding
morphism of $\mathrm{id}_{M^{D\left(  2\right)  }}$\ in $\mathrm{Hom}%
_{\mathcal{K}}\left(  D\left(  2\right)  \times M^{D\left(  2\right)
},M\right)  $\ is no other than the evaluation morphism $D\left(  2\right)
\times M^{D\left(  2\right)  }\rightarrow M$, and the corresponding morphism
of $\mathrm{id}_{M^{D}}$\ in $\mathrm{Hom}_{\mathcal{K}}\left(  D\times
M^{D},M\right)  $\ is no other than the evaluation morphism $D\times
M^{D}\rightarrow M$, Therefore both the evaluation morphisms $D\left(
2\right)  \times M^{D\left(  2\right)  }\rightarrow M$\ and $D\times
M^{D}\rightarrow M$\ yield the same morphism in $\mathrm{Hom}_{\mathcal{K}%
}\left(  D\times M^{D\left(  2\right)  },M\right)  $\ by application of their
adjacent vertical arrows, which is tantamount to the commutativity of the
diagram $\fbox{3}$. We have just established the commutativity of the whole
diagram (\ref{t5.1.1}). In particular, the outer hexagon of the diagram
(\ref{t5.1.1})\ is commutative, which means that the diagram%
\begin{equation}%
\begin{tabular}
[c]{lll}%
$D\times\mathbb{R}\times M^{D\left(  2\right)  }$ & $\longrightarrow$ &
$D\left(  2\right)  \times\mathbb{R}\times M^{D\left(  2\right)  }$\\
$\downarrow$ & $\circlearrowleft$ & $\downarrow$\\
$D\times\mathbb{R}\times M^{D}$ & $\longrightarrow$ & $M$%
\end{tabular}
\ \ \ \ \ \ \ \ \ \label{t5.1.4}%
\end{equation}
is commutative, where the two horizontal arrows are%
\begin{align*}
\mathbf{D}\left(
\begin{array}
[c]{c}%
\left(  X,X\right)  /\left(  X^{2}\right)  \leftarrow\\
\left(  X,Y\right)  /\left(  X^{2},Y^{2},XY\right)
\end{array}
\right)  \times\mathbb{R}\times M^{D\left(  2\right)  }  &  :\\
D\times\mathbb{R}\times M^{D\left(  2\right)  }  &  =\mathbf{D}\left(
\mathbf{R}\left[  X\right]  /\left(  X^{2}\right)  \right)  \times
\mathbb{R}\times M^{D\left(  2\right)  }\\
&  \rightarrow\\
\mathbf{D}\left(  \mathbf{R}\left[  X,Y\right]  /\left(  X^{2},Y^{2}%
,XY\right)  \right)  \times\mathbb{R}\times M^{D\left(  2\right)  }  &
=D\left(  2\right)  \times\mathbb{R}\times M^{D\left(  2\right)  }\\
\widehat{\psi}_{1}  &  :D\times\mathbb{R}\times M^{D}\rightarrow M
\end{align*}
from top to bottom, and the two vertical arrows are%
\begin{align*}
D\times\mathbb{R}\times\varphi &  :D\times\mathbb{R}\times M^{D\left(
2\right)  }\rightarrow D\times\mathbb{R}\times M^{D}\\
\widehat{\psi}_{2}  &  :D\left(  2\right)  \times\mathbb{R}\times M^{D\left(
2\right)  }\rightarrow M
\end{align*}
from left to right.

\item The following is a commutative diagram of parametrized adjunction (cf.
Theorem 3 in \S IV.7 of \cite{ma2}):%
\begin{equation}%
\begin{tabular}
[c]{lll}%
$\mathrm{Hom}_{\mathcal{K}}\left(  D\left(  2\right)  \times\mathbb{R}\times
M^{D\left(  2\right)  },M\right)  $ & $\cong$ & $\mathrm{Hom}_{\mathcal{K}%
}\left(  \mathbb{R}\times M^{D\left(  2\right)  },M^{D\left(  2\right)
}\right)  $\\
$\downarrow$ & $\circlearrowleft$ & $\downarrow$\\
$\mathrm{Hom}_{\mathcal{K}}\left(  D\times\mathbb{R}\times M^{D\left(
2\right)  },M\right)  $ & $\cong$ & $\mathrm{Hom}_{\mathcal{K}}\left(
\mathbb{R}\times M^{D\left(  2\right)  },M^{D}\right)  $\\
$\uparrow$ & $\circlearrowleft$ & $\uparrow$\\
$\mathrm{Hom}_{\mathcal{K}}\left(  D\times\mathbb{R}\times M^{D},M\right)  $ &
$\cong$ & $\mathrm{Hom}_{\mathcal{K}}\left(  \mathbb{R}\times M^{D}%
,M^{D}\right)  $%
\end{tabular}
\ \ \ \ \ \ \ \ \ \ \ \ \ \ \ \ \ \label{t5.1.3}%
\end{equation}
where the left two vertical arrows are%
\begin{align*}
\mathrm{Hom}_{\mathcal{K}}\left(  \left(
\begin{array}
[c]{c}%
\mathbf{D}\left(
\begin{array}
[c]{c}%
\left(  X,X\right)  /\left(  X^{2}\right)  \leftarrow\\
\left(  X,Y\right)  /\left(  X^{2},Y^{2},XY\right)
\end{array}
\right)  :\\
D=\mathbf{D}\left(  \mathbf{R}\left[  X\right]  /\left(  X^{2}\right)
\right)  \rightarrow\\
\mathbf{D}\left(  \mathbf{R}\left[  X,Y\right]  /\left(  X^{2},Y^{2}%
,XY\right)  \right)  =D\left(  2\right)
\end{array}
\right)  \times\mathbb{R}\times M^{D\left(  2\right)  },M\right)   &  :\\
\mathrm{Hom}_{\mathcal{K}}\left(  D\left(  2\right)  \times\mathbb{R}\times
M^{D\left(  2\right)  },M\right)   &  \rightarrow\\
&  \mathrm{Hom}_{\mathcal{K}}\left(  D\times\mathbb{R}\times M^{D\left(
2\right)  },M\right)  \\
\mathrm{Hom}_{\mathcal{K}}\left(  D\times\mathbb{R}\times\varphi,M\right)   &
:\\
\mathrm{Hom}_{\mathcal{K}}\left(  D\times\mathbb{R}\times M^{D},M\right)   &
\rightarrow\\
&  \mathrm{Hom}_{\mathcal{K}}\left(  D\times\mathbb{R}\times M^{D\left(
2\right)  },M\right)
\end{align*}
from top to bottom, while the right vertical arrows are%
\begin{align*}
\mathrm{Hom}_{\mathcal{K}}\left(  \mathbb{R}\times M^{D\left(  2\right)
},\varphi\right)   &  :\mathrm{Hom}_{\mathcal{K}}\left(  \mathbb{R}\times
M^{D\left(  2\right)  },M^{D\left(  2\right)  }\right)  \rightarrow
\mathrm{Hom}_{\mathcal{K}}\left(  \mathbb{R}\times M^{D\left(  2\right)
},M^{D}\right)  \\
\mathrm{Hom}_{\mathcal{K}}\left(  \mathbb{R}\times\varphi,M^{D}\right)   &
:\mathrm{Hom}_{\mathcal{K}}\left(  \mathbb{R}\times M^{D},M^{D}\right)
\rightarrow\mathrm{Hom}_{\mathcal{K}}\left(  \mathbb{R}\times M^{D\left(
2\right)  },M^{D}\right)
\end{align*}
from top to bottom. Choose%
\begin{align*}
\widehat{\psi}_{2} &  \in\mathrm{Hom}_{\mathcal{K}}\left(  D\left(  2\right)
\times\mathbb{R}\times M^{D\left(  2\right)  },M\right)  \\
\widehat{\psi}_{1} &  \in\mathrm{Hom}_{\mathcal{K}}\left(  D\times
\mathbb{R}\times M^{D},M\right)
\end{align*}
on the left of the diagram.(\ref{t5.1.3}). Then both yield the same morphism
in $\mathrm{Hom}_{\mathcal{K}}\left(  D\times\mathbb{R}\times M^{D\left(
2\right)  },M\right)  $\ by application of their adjacent vertical arrows by
dint of the commutativity of the diagram (\ref{t5.1.4}). The corresponding
morphism of $\widehat{\psi}_{2}$\ in $\mathrm{Hom}_{\mathcal{K}}\left(
\mathbb{R}\times M^{D\left(  2\right)  },M^{D\left(  2\right)  }\right)  $\ is
$\psi_{2}$, and the corresponding morphism of $\widehat{\psi}_{1}$\ in
$\mathrm{Hom}_{\mathcal{K}}\left(  \mathbb{R}\times M^{D},M^{D}\right)  $\ is
$\psi_{1}$, Therefore both $\psi_{2}$\ and $\psi_{1}$\ yield the same morphism
in $\mathrm{Hom}_{\mathcal{K}}\left(  \mathbb{R}\times M^{D\left(  2\right)
},M^{D}\right)  $\ by application of their adjacent vertical arrows, which is
tantamount to the commutativity of the following diagram:%
\begin{equation}%
\begin{tabular}
[c]{lll}%
$\mathbb{R}\times M^{D\left(  2\right)  }$ & $\longrightarrow$ & $M^{D\left(
2\right)  }$\\
$\downarrow$ & $\circlearrowleft$ & $\downarrow$\\
$\mathbb{R}\times M^{D}$ & $\longrightarrow$ & $M^{D}$%
\end{tabular}
\ \ \ \ \ \ \ \ \ \ \ \ \label{t5.1.5}%
\end{equation}
where the two horizontal arrows are%
\begin{align*}
\psi_{2} &  :\mathbb{R}\times M^{D\left(  2\right)  }\rightarrow M^{D\left(
2\right)  }\\
\psi_{1} &  :\mathbb{R}\times M^{D}\rightarrow M^{D}%
\end{align*}
from top to bottom, and the two vertical arrows are%
\begin{align*}
\mathbb{R}\times\varphi &  :\mathbb{R}\times M^{D\left(  2\right)
}\rightarrow\mathbb{R}\times M^{D}\\
\varphi &  :M^{D\left(  2\right)  }\rightarrow M^{D}%
\end{align*}
from left to right. We have just established the distibutivity of scalar
multiplication over addition.
\end{enumerate}
\end{itemize}
\end{proof}

\section{\label{s6}Concluding Remarks}

Weilology began with Andr\'{e} Weil's algebraic treatment of nilpotent
infinitesimals \cite{we1}. Its second step is synthetic differential geometry
\cite{koc1}and the study of Weil functors of Czech geometers \cite{kol1}. Its
third step is the author's axiomatic differential geometry (\cite{ni1}%
-\cite{ni8}). Now we have its final form in this paper.

A subsequent paper is devoted to fixing the syntax of Weil categories after
the manner of \cite{bel1}, under which we can develop axiomatic differential
geometry naively (i.e., without tears), just as Ren\'{e} Lavendhomme did for
synthetic differential geometry \cite{la1}. 

Another important point is that we can investigate Weilology for
supergeometry, braided geometry, noncommutative geometry, homotopical
differential geometry, arithmetical differential geometry and so on in the
same vein, which is the topic of subsequent papers.

\end{document}